\documentclass{article}
\usepackage[english]{babel}
\usepackage[latin1]{inputenc}
\usepackage{amsmath,amsthm}
\usepackage{amssymb}
\usepackage{mathtools}
\usepackage{mathptmx}
\usepackage{url}
\usepackage[backref=page, a4paper, breaklinks=true]{hyperref}
\hypersetup{pdfpagemode=UseNone, pdfstartview={FitH}}
\usepackage{breakurl}
\usepackage{cleveref}
\newtheorem{theorem}{Theorem}[section]
\newtheorem{remark}{Remark}[section]
\newtheorem{definition}{Definition}[section]
\newtheorem{proposition}{Proposition}[section]
\newtheorem{lemma}{Lemma}[section]
\newtheorem{corollary}{Corollary}[section]
\newtheorem{problem}{Problem}[section]

\newcommand{\lag}{\left \langle}
\newcommand{\rog}{\right \rangle}
\newcommand{\df}{\mathrel{\mathop:}=}

\begin{document}
\title{The variety generated by all the ordinal sums of perfect MV-chains}
\author{Matteo Bianchi\\{\small Dipartimento di Informatica e Comunicazione}\\{\small Università degli Studi di Milano}\\{\footnotesize \texttt{\href{mailto:matteo.bianchi@unimi.it}{matteo.bianchi@unimi.it}}}}
\date{}
\maketitle
\begin{quote}
\small \it ``To my friend Erika, and to her invaluable talent in finding surprisingly deep connections among poetry, art, philosophy and logic.''
\end{quote}
\begin{abstract}
We present the logic BL$_\text{Chang}$, an axiomatic extension of BL (see \cite{haj}) whose corresponding algebras form the smallest variety containing all the ordinal sums of perfect MV-chains. 
We will analyze this logic and the corresponding algebraic semantics in the propositional and in the first-order case. As we will see, moreover, the variety of BL$_\text{Chang}$-algebras will be strictly connected to the one generated by Chang's MV-algebra (that is, the variety generated by all the perfect MV-algebras): we will also give some new results concerning these last structures and their logic.
\end{abstract}
\section{Introduction}
MV-algebras were introduced in \cite{chang} as the algebraic counterpart of \L ukasie\-wicz (infinite-valued) logic. During the years these structures have been intensively studied (for a hystorical overview, see \cite{cig}): the book \cite{mun} is a reference monograph on this topic.

Perfect MV-algebras were firstly studied in \cite{dnl2} as a refinement of the notion of local MV-algebras: this analysis was expanded in \cite{dnl1}, where it was also shown that the class of perfect MV-algebras $Perf(MV)$ does not form a variety, and the variety generated by $Perf(MV)$ is also generated by Chang's MV-algebra (see \Cref{sec:mvcha} for the definition). Further studies, about this variety and the associated logic have been done in \cite{bdg, bdg1}. 

On the other side, Basic Logic BL and its correspondent variety, BL-algebras, were introduced in \cite{haj}: \L ukasiewicz logic results to be one of the axiomatic extensions of BL and MV-algebras can also be defined as a subclass of BL-algebras. Moreover, the connection between MV-algebras and BL-algebras is even stronger: in fact, as shown in \cite{am}, every ordinal sum of MV-chains is a BL-chain.

For these reasons one can ask if there is a variety of BL-algebras whose chains are (isomorphic to) ordinal sums of perfect MV-chains: even if the answer to this question is negative, we will present the smallest variety (whose correspondent logic is called BL$_\text{Chang}$) containing this class of BL-chains. 

As we have anticipated in the abstract, there is a connection between the variety of BL$_\text{Chang}$-algebras and the one generated by Chang's MV-algebra. In fact the first-one is axiomatized (over the variety of BL-algebras) with an equation that, over MV-algebras, is equivalent to the one that axiomatize the variety generated by Chang MV-algebras: however, the two equations are \emph{not} equivalent, over BL.
\paragraph*{}
The paper is structured as follows: in \Cref{sec:prelim} we introduce the necessary logical and algebraic background: moreover some basic results about perfect MV-algebras and other structures will be listed. In \Cref{sec:blcha} we introduce the main theme of the article: the variety of BL$_\text{Chang}$ and the associated logic. The analysis will be done in the propositional case: completeness results, algebraic and logical properties and also some results about the variety generated by Chang's MV-algebra.  
We conclude with \Cref{sec:first}, where we will analyze the first-order versions of BL$_\text{Chang}$ and \L$_\text{Chang}$: for the first-one the completeness results will be much more negative. 
\paragraph{}
To conclude, we list the main results.
\begin{itemize}
\item BL$_\text{Chang}$ enjoys the finite strong completeness (but not the strong one) w.r.t. $\omega\mathcal{V}$, where $\omega\mathcal{V}$ represents the ordinal sum of $\omega$ copies of the disconnected rotation of the standard cancellative hoop. 
\item \L$_\text{Chang}$ (the logic associated to the variety generated by Chang's MV-algebra) enjoys the finite strong completeness (but not the strong one) w.r.t. $\mathcal{V}$, $\mathcal{V}$ being the disconnected rotation of the standard cancellative hoop.
\item There are two BL-chains $\mathcal{A}, \mathcal{B}$ that are strongly complete w.r.t., respectively \L$_\text{Chang}$ and BL$_\text{Chang}$.
\item Every \L$_\text{Chang}$-chain that is strongly complete w.r.t. \L$_\text{Chang}$ is also stron\-gly complete w.r.t \L$_\text{Chang}\forall$.
\item There is no BL$_\text{Chang}$-chain to be complete w.r.t. BL$_\text{Chang}\forall$.   
\end{itemize} 
\section{Preliminaries}\label{sec:prelim}
\subsection{Basic concepts}
Basic Logic BL was introduced by P. H\'{a}jek in \cite{haj}. It is based over the connectives $\{\&,\to,\bot\}$ and a denumerable set of variables $VAR$. The formulas are defined inductively, as usual (see \cite{haj} for details).

Other derived connectives are the following.

{\em negation}: $\neg \varphi \df \varphi\to\bot$;
{\em verum} or {\em top}: $\top \df\neg\bot$;
{\em meet}: $\varphi\land\psi \df \varphi\&(\varphi\to\psi)$;
{\em join}: $\varphi\vee\psi \df ((\varphi\to\psi)\to\psi)\land((\psi\to\varphi)\to\varphi)$.
\paragraph*{}
\noindent BL is axiomatized as follows.
\begin{align}
\tag{A1}&(\varphi \rightarrow \psi)\rightarrow ((\psi\rightarrow \chi)\rightarrow(\varphi\rightarrow \chi))\\
\tag{A2}&(\varphi\&\psi)\rightarrow \varphi\\
\tag{A3}&(\varphi\&\psi)\rightarrow(\psi\&\varphi)\\
\tag{A4}&(\varphi\&(\varphi\to\psi))\to(\psi\& (\psi\to\varphi))\\
\tag{A5a}&(\varphi\rightarrow(\psi\rightarrow\chi))\rightarrow((\varphi\&\psi)\rightarrow \chi)\\
\tag{A5b}&((\varphi\&\psi)\rightarrow \chi)\rightarrow(\varphi\rightarrow(\psi\rightarrow\chi))\\
\tag{A6}&((\varphi\rightarrow\psi)\rightarrow\chi)\rightarrow(((\psi\rightarrow\varphi)\rightarrow\chi)\rightarrow\chi)\\
\tag{A7}&\bot\rightarrow\varphi
\end{align}
{\em Modus ponens} is the only inference rule:
\begin{equation}
\tag{MP}\frac{\varphi\quad\varphi\to\psi}{\psi}.
\end{equation}
Among the extensions of BL (logics obtained from it by adding other axioms) there is the well known \L ukasiewicz (infinitely-valued) logic \L, that is, BL plus
\begin{equation}
\tag*{(INV)}\neg\neg\varphi\to\varphi.
\end{equation}
On \L ukasiewicz logic we can also define a strong disjunction connective (in the following sections, we will introduce a strong disjunction connective, for BL, that will be proved to be equivalent to the following, over \L)
\begin{equation*}
\varphi\curlyvee\psi\df\neg(\neg\varphi\&\neg\psi).
\end{equation*}
The notations $\varphi^n$ and $n\varphi$ will indicate $\underbrace{\varphi\&\dots\&\varphi}_{n\text{ times}}$ and $\underbrace{\varphi\curlyvee\dots\curlyvee\varphi}_{n\text{ times}}$.

Given an axiomatic extension L of BL, a formula $\varphi$ and a theory $T$ (a set of formulas), the notation $T\vdash_L\varphi$ indicates that there is a proof of $\varphi$ from the axioms of L and the ones of $T$. The notion of proof is defined like in classical case (see \cite{haj}).
\paragraph*{}
We now move to the semantics: for all the unexplained notions of universal algebra, we refer to \cite{bs, gra}.
\begin{definition}
A BL-algebra is an algebraic structure of the form $\mathcal{A}=\lag A,*,\Rightarrow,\sqcap,\sqcup,0,1\rog$ such that
\begin{itemize}
\item $\lag A,\sqcap,\sqcup,0,1 \rog$ is a bounded lattice, where $0$ is the bottom and $1$ the top element.
\item $\lag A,*,1\rog$ is a commutative monoid.
\item $\lag *,\Rightarrow\rog$ forms a residuated pair,\index{residuated pair} i.e.
\begin{equation}
\tag*{(res)}z*x\leq y\quad\text{iff}\quad z\leq x\Rightarrow y, \label{eq:res}
\end{equation}
it can be shown that the only operation that satisfies \ref{eq:res} is $x\Rightarrow y=\max\{z:\ z*x\leq y\}$.
\item $\mathcal{A}$ satisfies the following equations
\begin{align}
\tag*{(pl)}&(x\Rightarrow y)\sqcup(y\Rightarrow x)=1\\
\tag*{(div)}&x\sqcap y=x*(x\Rightarrow y).
\end{align}
\end{itemize}
Two important types of BL-algebras are the followings.
\begin{itemize}
\item A BL-chain is a totally ordered BL-algebra.
\item A standard BL-algebra is a BL-algebra whose support is $[0,1]$.
\end{itemize}
\end{definition}
Notation: in the following, with $\sim x$ we will indicate $x\Rightarrow 0$.
\begin{definition}
An MV-algebra is a BL-algebra satisfying
\begin{equation}
\tag*{(inv)}x=\sim\sim x.
\end{equation}
A well known example of MV-algebra is the standard MV-algebra $[0,1]_\text{\L}=\\\lag [0,1], *,\Rightarrow, \min, \max, 0, 1\rog$, where $x*y=\max(0,x+y-1)$ and $x\Rightarrow y=\min(1, 1-x+y)$.
\end{definition}
In every MV-algebra we define the algebraic equivalent of $\curlyvee$, that is
\begin{equation*}
x\oplus y\df\sim(\sim x * \sim y).
\end{equation*}
The notations (where $x$ is an element of some BL-algebra) $x^n$ and $nx$ will indicate $\underbrace{x*\dots *x}_{n\text{ times}}$ and $\underbrace{x\oplus\dots\oplus x}_{n\text{ times}}$.

Given a BL-algebra $\mathcal{A}$, the notion of $\mathcal{A}$-evaluation is defined in a truth-functional way (starting from a map $v:\, VAR\to A$, and extending it to formulas), for details see \cite{haj}.

Consider a BL-algebra $\mathcal{A}$, a theory $T$ and a formula $\varphi$. With $\mathcal{A}\models \varphi$ ($\mathcal{A}$ is a model of $\varphi$) we indicate that $v(\varphi)=1$, for every $\mathcal{A}$-evaluation $v$; $\mathcal{A}\models T$ denotes that $\mathcal{A}\models \psi$, for every $\psi\in T$. Finally, the notation $T\models_\mathcal{A}\varphi$ means that if $\mathcal{A}\models T$, then $\mathcal{A}\models \varphi$.

A BL-algebra $\mathcal{A}$ is called L-algebra, where L is an axiomatic extension of BL, whenever $\mathcal{A}$ is a model for all the axioms of L.
\begin{definition}
Let L be an axiomatic extension of BL and $K$ a class of L-algebras. We say that L is strongly
complete (respectively: finitely strongly complete, complete) with respect to $K$
if for every set $T$ of formulas (respectively, for every finite set $T$ of formulas, for
$T=\emptyset$) and for every formula $\varphi$ we have
\begin{equation*}
T\vdash_L\varphi\quad \text{iff}\quad T\models_K\varphi.
\end{equation*}
\end{definition}
\subsection{Perfect MV-algebras, hoops and disconnected rotations}\label{sec:mvcha}
We recall that Chang's \emph{MV}-algebra (\cite{chang}) is a BL-algebra of the form
\begin{equation*}
C=\lag \{a_n:\ n\in\mathbb{N}\}\cup\{b_n:\ n\in\mathbb{N}\}, *, \Rightarrow, \sqcap,\sqcup, b_0, a_0\rog.
\end{equation*}
Where for each $n,m\in \mathbb{N}$, it holds that $b_n<a_m$, and, if $n<m$, then $a_m<a_n,\ b_n<b_m$; moreover $a_0=1,\ b_0=0$ (the top and the bottom element).
\smallskip

The operation $*$ is defined as follows, for each $n,m\in \mathbb{N}$:
\begin{equation*}
b_n*b_m=b_0,\ b_n*a_m=b_{\max(0,n-m)},\ a_n*a_m=a_{n+m}.
\end{equation*}
\begin{definition}[\cite{dnl2}]
Let $\mathcal{A}$ be an MV-algebra and let $x\in\mathcal{A}$: with $ord(x)$ we mean the least (positive) natural $n$ such that $x^n=0$. If there is no such $n$, then we set $ord(x)=\infty$.
\begin{itemize}
\item An MV-algebra is called \emph{local}\footnote{Usually, the local MV-algebras are defined as MV-algebras having a unique (proper) maximal ideal. In \cite{dnl2}, however, it is shown that the two definitions are equivalent. We have preferred the other definition since it shows in a more transparent way that perfect MV-algebras are particular cases of local MV-algebras.} if for every element $x$ it holds that \\$ord(x)<\infty$ or $ord(\sim x)<\infty$.
\item An MV-algebra is called \emph{perfect} if for every element $x$ it holds that $ord(x)<\infty$ iff $ord(\sim x)=\infty$.
\end{itemize}
\end{definition}
An easy consequence of this definition is that every perfect MV-algebra cannot have a negation fixpoint.

With $Perfect(MV)$ and $Local(MV)$ we will indicate the class of perfect and local MV-algebras. Moreover, given a BL-algebra $\mathcal{A}$, with $\mathbf{V}(\mathcal{A})$ we will denote the variety generated by $\mathcal{A}$.
\begin{theorem}[\cite{dnl2}]
Every MV-chain is local.
\end{theorem}
Clearly there are local MV-algebras that are not perfect: $[0,1]_\text{\L}$ is an example.

Now, in \cite{dnl1} it is shown that
\begin{theorem}\label{teo:perfcha}
\begin{itemize}
\item[]
\item $\mathbf{V}(C)=\mathbf{V}(Perfect(MV))$,
\item $Perfect(MV)=Local(MV)\cap\mathbf{V}(C)$.
\end{itemize}
\end{theorem}
It follows that the class of chains in $\mathbf{V}(C)$ coincides with the one of perfect MV-chains. Moreover
\begin{theorem}[\cite{dnl1}]\label{teo:chaeq}
An MV-algebra is in the variety $\mathbf{V}(C)$ iff it satisfies the equation $(2x)^2=2(x^2)$.
\end{theorem}
As shown in \cite{bdg}, the logic correspondent to this variety is axiomatized as {\L} plus $(2\varphi)^2\leftrightarrow 2(\varphi^2)$: we will call it {\L}$_\text{Chang}$.
\paragraph*{}
We now recall some results about hoops
\begin{definition}[\cite{f, bf}]
A \emph{hoop} is a structure $\mathcal{A}=\langle A, *
,\Rightarrow ,1\rangle $ such that $\langle A, * ,1\rangle $ is a
commutative monoid, and $\Rightarrow $ is a binary operation such that
\[
x\Rightarrow x=1,\hspace{0.5cm}x\Rightarrow (y\Rightarrow z)=(x *
y)\Rightarrow z\hspace{0.5cm}\mathrm{and}\hspace{0.5cm}x * (x\Rightarrow
y)=y * (y\Rightarrow x).
\]
\end{definition}

In any hoop, the operation $\Rightarrow $ induces a partial order $\le $
defined by $x\le y$ iff $x\Rightarrow y=1$. Moreover, hoops are precisely
the partially ordered commutative integral residuated monoids (pocrims) in
which the meet operation $\sqcap$ is definable by $x\sqcap y=x *
(x\Rightarrow y)$. Finally, hoops satisfy the following divisibility
condition:
\begin{equation}
\tag*{(div)}\text{If } x \le y, \text{ then there is an element } z\text{ such
that }z * y=x.
\end{equation}
We recall a useful result.
\begin{definition}
Let $\mathcal{A}$ and $\mathcal{B}$ be two algebras of the same language. Then we say that
\begin{itemize}
\item $\mathcal{A}$ is a partial subalgebra of $\mathcal{B}$ if $A\subseteq B$ and the operations of $\mathcal{A}$ are the ones of $\mathcal{A}$ restricted to $A$. Note that $A$ could not be closed under these operations (in this case these last ones will be undefined over some elements of $A$): in this sense $\mathcal{A}$ is a partial subalgebra.
\item $\mathcal{A}$ is partially embeddable into $\mathcal{B}$ when every finite partial subalgebra of $\mathcal{A}$ is embeddable into $\mathcal{B}$. Generalizing
this notion to classes of algebras, we say that a class $K$ of algebras is partially embeddable into a
class $M$ if every finite partial subalgebra of a member of $K$ is embeddable into a member of $M$.
\end{itemize}
\end{definition}
\begin{definition}\label{def:bound}
A \emph{bounded} hoop is a hoop with a minimum element; conversely, an \emph{unbounded} hoop is a hoop without minimum. 

Let $\mathcal{A}$ be a bounded hoop with minimum $a$: with $\mathcal{A}^+$ we mean the (partial) subalgebra of $\mathcal{A}$ defined over the universe $A^+=\{x\in A:\, x>x\Rightarrow a\}$.

A hoop is Wajsberg iff it satisfies the equation $(x\Rightarrow y)\Rightarrow y=(y\Rightarrow x)\Rightarrow x$.

A hoop is cancellative iff it satisfies the equation $x=y\Rightarrow(x*y)$.
\end{definition}
\begin{proposition}[\cite{f, bf, afm}]\label{prop:canc}
Every cancellative hoop is Wajsberg. Totally ordered cancellative hoops coincide with unbounded totally ordered Wajsberg hoops, whereas bounded Wajsberg hoops coincide with (the $0$-free reducts of) MV-algebras.
\end{proposition}
We now recall a construction introduced in \cite{jen} (and also used in \cite{eghm, neg}), called \emph{disconnected rotation}.

\begin{definition}
Let $\mathcal{A}$ be a cancellative hoop. We define an algebra, $\mathcal{A}^*$, called the
\emph{disconnected rotation} of $\mathcal{A}$, as follows. Let $\mathcal{A}\times\{0\}$ be a disjoint copy of A. For every $a\in A$ we
write $a'$ instead of $\langle a, 0\rangle$. Consider $\lag A'= \{a' : a \in A\}, \leq\rog$
with the inverse order and let $A^*\df A\cup A'$. We extend these orderings to an order in $A^*$ by putting $a' < b$
for every $a,b\in A$. Finally, we take the following operations in $A^*$: $1\df 1_\mathcal{A}$, $0\df 1'$, $\sqcap_{\mathcal{A}^*}, \sqcup_{\mathcal{A}^*}$ as the meet and the join with respect to the order over $A^*$. Moreover,
\begin{align*}
\sim_{\mathcal{A}^*} a\df&\begin{cases}
a'&\text{if }a\in A\\
b &\text{if }a=b'\in A'
         \end{cases}\\
a*_{\mathcal{A}^*}b\df&\begin{cases}
a*_\mathcal{A} b&\text{if }a, b\in A\\
\sim_{\mathcal{A}^*}(a\Rightarrow_{\mathcal{A}} \sim_{\mathcal{A}^*}b)&\text{if }a\in A, b\in A'\\
\sim_{\mathcal{A}^*}(b\Rightarrow_{\mathcal{A}} \sim_{\mathcal{A}^*}a)&\text{if }a\in A', b\in A\\
0&\text{if }a, b\in A'
         \end{cases}\\
a\Rightarrow_{\mathcal{A}^*}b\df&\begin{cases}
a\Rightarrow_\mathcal{A} b&\text{if }a, b\in A\\
\sim_{\mathcal{A}^*}(a *_{\mathcal{A}^*} \sim_{\mathcal{A}^*}b)&\text{if }a\in A, b\in A'\\
1&\text{if }a\in A', b\in A\\
(\sim_{\mathcal{A}^*}b) \Rightarrow_{\mathcal{A}} (\sim_{\mathcal{A}^*}a)&\text{if }a, b\in A'.
         \end{cases}
\end{align*}
\end{definition}

\begin{theorem}[{\cite[theorem 9]{neg}}]\label{teo:perfrot}
Let $\mathcal{A}$ be an MV-algebra. The followings are equivalent:
\begin{itemize}
 \item A is a perfect MV-algebra.
 \item A is isomorphic to the disconnected rotation of a cancellative hoop.
\end{itemize}
\end{theorem}
To conclude the section, we present the definition of ordinal sum.
\begin{definition}[\cite{am}]
Let $\langle I,\le \rangle $ be a totally ordered set with minimum $0$.
For all $i\in I$, let $\mathcal{A}_{i}$ be a hoop such that for $i\ne j$, $%
A_{i}\cap A_{j}=\{1\}$, and assume that $\mathcal{A}_0$ is bounded.
Then $\bigoplus_{i\in I}{\mathcal{A}}_{i}$ (the \emph{ordinal sum} of the
family $({\mathcal{A}}_{i})_{i\in I}$) is the structure whose base set is $%
\bigcup_{i\in I}A_{i}$, whose bottom is the minimum of $\mathcal{A}_0$,
whose top is $1$, and whose operations are
\begin{align*}
x\Rightarrow y&=\begin{cases}
x\Rightarrow ^{{\mathcal{A}}_{i}}y & \mathrm{if} \,\,x,y\in A_{i} \\
y & \mathrm{if}\,\, \exists i>j(x\in A_{i}\,\,\mathrm{and}\,\,y\in A_{j}) \\
1 & \mathrm{if}\,\,\exists i<j(x\in A_{i}\setminus \{1\}\,\,\mathrm{and%
}\,\,y\in A_{j})
\end{cases}\\
x* y&=\begin{cases}
x * ^{{\mathcal{A}}_{i}}y & \mathrm{if}\,\,x,y\in A_{i} \\
x & \mathrm{if}\,\,\exists i<j(x\in A_{i}\setminus\{1\},\,\,y\in A_{j})\\
y & \mathrm{if}\,\,\exists i<j(y\in A_{i}\setminus\{1\},\,x\in A_{j})
\end{cases}
\end{align*}
When defining the ordinal sum $\bigoplus_{i\in I}{\mathcal{A}}_{i}$ we will
tacitly assume that whenever the condition $A_{i}\cap A_{j}=\left\{
1\right\} $ is not satisfied for all $i,j\in I$ with $i\neq j$, we will
replace the $\mathcal{A}_{i}$ by isomorphic copies satisfying such condition. Moreover if all $\mathcal{A}_i$'s are isomorphic to some $\mathcal{A}$, then we will write $I\mathcal{A}$, instead of $\bigoplus_{i \in I}\mathcal{A}_{i}$. Finally, the ordinal sum of two hoops $\mathcal{A}$ and $\mathcal{B}$ will be denoted by $\mathcal{A}\oplus\mathcal{B}$.
\end{definition}
Note that, since every bounded Wajsberg hoop is the $0$-free reduct of an MV-algebra, then the previous definition also works with these structures.
\begin{theorem}[{\cite[theorem 3.7]{am}}]\label{teo:am}
Every BL-chain is isomorphic to an ordinal sum whose first component is an MV-chain and the others are totally ordered Wajsberg hoops.
\end{theorem}
Note that in \cite{bus} it is presented an alternative and simpler proof of this result.
\section{The variety of BL$_\text{Chang}$-algebras}\label{sec:blcha}
Consider the following connective
\begin{equation*}
\varphi\veebar\psi\df((\varphi \rightarrow (\varphi \& \psi ))\rightarrow\psi )\land ((\psi \rightarrow (\varphi \& \psi ))\rightarrow \varphi )
\end{equation*}
Call $\uplus$ the algebraic operation, over a BL-algebra, corresponding to $\veebar$; we have that
\begin{lemma}\label{lem:disgeq}
In every MV-algebra the following equation holds
\begin{equation*}
x\uplus y=x\oplus y.
\end{equation*}
\end{lemma}
\begin{proof}
It is easy to check that $x\uplus y=x\oplus y$, over $[0,1]_{MV}$, for every $x,y\in [0,1]$.
\end{proof}
We now analyze this connective in the context of Wajsberg hoops.
\begin{proposition}\label{prop:disj}
Let $\mathcal{A}$ be a linearly ordered Wajsberg hoop. Then
\begin{itemize}
\item If $\mathcal{A}$ is unbounded (i.e. a cancellative hoop), then $x\uplus y=1$, for every $x,y\in\mathcal{A}$.
\item If $\mathcal{A}$ is bounded, let $a$ be its minimum. Then, by defining $\sim x\df x\Rightarrow a$ and $x\oplus y=\sim(\sim x*\sim y)$ we have that $x\oplus y=x\uplus y$, for every $x,y\in\mathcal{A}$
\end{itemize}
\end{proposition}
\begin{proof}
An easy check.
\end{proof}
Now, since the variety of cancellative hoops is generated by its linearly ordered members (see \cite{eghm}), then we have that
\begin{corollary}\label{cor:disjcanc}
The equation $x\uplus y=1$ holds in every cancellative hoop.
\end{corollary}
We now characterize the behavior of $\uplus$ for the case of BL-chains.
\begin{proposition}\label{prop:disg}
Let $\mathcal{A}=\bigoplus_{i\in I}\mathcal{A}_i$ be a BL-chain. Then
\begin{equation*}
x\uplus y=\begin{cases}
x\oplus y,&\text{if }x,y\in \mathcal{A}_i\text{ and }\mathcal{A}_i\text{ is bounded }\\
1,&\text{if }x,y\in \mathcal{A}_i\text{ and }\mathcal{A}_i\text{ is unbounded }\\
\max(x,y),&\text{otherwise}.
\end{cases}
\end{equation*}
for every $x,y\in\mathcal{A}$.
\end{proposition}
\begin{proof}
If $x,y$ belong to the same component of $\mathcal{A}$, then the result follows from \Cref{lem:disgeq} and \Cref{prop:disj}. For the case in which $x$ and $y$ belong to different components of $\mathcal{A}$, this is a direct computation.
\end{proof}
\begin{remark}
From the previous proposition we can argue that $\uplus$ is a good approximation, for BL, of what that $\oplus$ represents for MV-algebras. Note that a similar operation was introduced in \cite{abm}: the main difference with respect to $\uplus$ is that, when $x$ and $y$ belong to different components of a BL-chain, then the operation introduced in \cite{abm} holds $1$.
\end{remark}
In the following, for every element $x$ of a BL-algebra, with the notation $\overline{n}x$ we will denote $\underbrace{x\uplus\dots\uplus x}_{n\ times}$; analogously $\overline{n}\varphi$ means $\underbrace{\varphi\veebar\dots\veebar\varphi}_{n\ times}$.
\begin{definition}
We define BL$_\text{Chang}$ as the axiomatic extension of BL, obtained by adding
\begin{equation}
\tag{cha}(\overline{2}\varphi)^2\leftrightarrow \overline{2}(\varphi^2).
\end{equation}
That is, writing it in extended form
\begin{equation*}
(\varphi^2\to(\varphi^2\&\varphi^2)\to\varphi^2)\leftrightarrow((\varphi\to\varphi^2)\to\varphi)^2.
\end{equation*}
\end{definition}
Clearly the variety corresponding to BL$_\text{Chang}$ is given by the class of BL-algebras satisfying the equation $(\overline{2}x)^2=\overline{2}(x^2)$.

Moreover,
\begin{definition}\label{def:pseudoperf}
We will call pseudo-perfect Wajsberg hoops those Wajsberg hoops satisfying the equation $(\overline{2}x)^2=\overline{2}(x^2)$.
\end{definition}
\begin{remark}\label{rem:2}
Thanks to \Cref{lem:disgeq} we have that
\begin{equation*}
\vdash_\text{\L}((\overline{2}\varphi)^2\leftrightarrow \overline{2}(\varphi^2))\leftrightarrow ((2\varphi)^2\leftrightarrow 2(\varphi^2)),
\end{equation*}
that is, if we add $(\overline{2}\varphi)^2\leftrightarrow \overline{2}(\varphi^2)$ or $(2\varphi)^2\leftrightarrow 2(\varphi^2)$ to {\L}, then we obtain the same logic \L$_\text{Chang}$.

These formulas, however are not equivalent over BL: see \Cref{rem:p0} for details.
\end{remark}
\begin{theorem}\label{teo:chainpswh}
Every totally ordered pseudo-perfect Wajsberg hoop is a totally ordered cancellative hoop or (the $0$-free reduct of) a perfect MV-chain.

More in general, the variety of pseudo-perfect Wajsberg hoops coincides with the class of the $0$-free subreducts of members of $\mathbf{V}(C)$.
\end{theorem}
\begin{proof}
In \cite{eghm} it is shown that the variety of Wajsberg hoops coincides with the class of the $0$-free subreducts of MV-algebras. The results easily follow from this fact and from \Cref{prop:canc}, \Cref{teo:chaeq} and \Cref{def:pseudoperf}.
\end{proof}
As a consequence, we have
\begin{theorem}\label{teo:hoopincl}
Let $\mathbb{WH}, \mathbb{CH}, ps\mathbb{WH}$ be, respectively, the varieties of Wajsberg hoops, cancellative hoops, pseudo-perfect Wajsberg hoops. Then we have that
\begin{equation*}
\mathbb{CH}\subset ps\mathbb{WH} \subset \mathbb{WH}.
\end{equation*}
\end{theorem}
\begin{proof}
An easy consequence of \Cref{teo:chainpswh}.

The first inclusion follows from the fact that $ps\mathbb{WH}$ contains all the totally ordered cancellative hoops and hence the variety generated by them. For the second inclusion note that, for example, the $0$-free reduct of $[0,1]_\text{\L}$ belongs to $\mathbb{WH}\setminus ps\mathbb{WH}$.
\end{proof}
We now describe the structure of BL$_\text{Chang}$-chains, with an analogous of the \Cref{teo:am} for BL-chains.
\begin{theorem}\label{teo:chainstruct}
Every BL$_\text{Chang}$-chain is isomorphic to an ordinal sum whose first component is a perfect MV-chain and the others are totally ordered pseudo-perfect Wajsberg hoops.

It follows that every ordinal sum of perfect MV-chains is a BL$_\text{Chang}$-chain.
\end{theorem}
\begin{proof}
Thanks to \Cref{teo:perfcha,teo:chaeq}, \Cref{rem:2} and \Cref{def:pseudoperf}, we have that every MV-chain (Wajsberg hoop) satisfying the equation $(\overline{2}x)^2=\overline{2}(x^2)$ is perfect (pseudo-perfect): using these facts and \Cref{prop:disg} we have that a BL-chain satisfies the equation $(\overline{2}x)^2=\overline{2}(x^2)$ iff it holds true in all the components of its ordinal sum. From these facts and \Cref{teo:am} we get the result.
\end{proof}
As a consequence, we obtain the following corollaries.
\begin{corollary}\label{cor:blchacont}
The variety of BL$_\text{Chang}$-algebras contains the ones of \\product-algebras and G\"{o}del-algebras: however it does not contains the variety of MV-algebras.
\end{corollary}
\begin{proof}
From the previous theorem it is easy to see that the variety of BL$_\text{Chang}$-algebras contains $[0,1]_\Pi$ and $[0,1]_G$, but not $[0,1]_\text{\L}$.
\end{proof}
\begin{corollary}
Every finite BL$_\text{Chang}$-chain is an ordinal sum of a finite number of copies of the two elements boolean algebra. Hence the class of finite BL$_\text{Chang}$-chains coincides with the one of finite G\"{o}del chains.
\end{corollary}
For this reason it is immediate to see that the finite model property does not hold for BL$_\text{Chang}$.

We conclude with the following remark.
\begin{remark}\label{rem:p0}
\begin{itemize}
\item One can ask if it is possible to axiomatize the class BL$_\text{perf}$ of BL-algebras, whose chains are the BL-algebras that are ordinal sum of perfect MV-chains: the answer, however, is negative. In fact, the class of bounded Wajsberg hoops does not form a variety: for example, it is easy to check that for every bounded pseudo-perfect Wajsberg hoop $\mathcal{A}$, its subalgebra  $\mathcal{A}^+$ (see \Cref{def:bound} ) forms a cancellative hoop. Hence BL$_\text{perf}$ cannot be a variety.

However, as we will see in \Cref{sec:propcompl}, the variety of BL$_\text{Chang}$-algebras is the ``best approximation'' of BL$_\text{perf}$, in the sense that it is the smallest variety to contain BL$_\text{perf}$.
\item In \cite{dnl4} (see also \cite{ct}) it is studied the variety, called $P_0$, generated by all the perfect BL-algebras (a BL-algebra $\mathcal{A}$ is perfect if, by calling $MV(\mathcal{A})$ the biggest subalgebra of $\mathcal{A}$ to be an MV-algebra, then $MV(\mathcal{A})$ is a perfect MV-algebra). $P_0$ is axiomatized with the equation
    \begin{equation}
    \tag*{($p_0$)}\sim((\sim(x^2))^2)=(\sim ((\sim x)^2))^2.\label{eq:p0}
    \end{equation}
    One can ask which is the relation between $P_0$ and the variety of BL$_\text{Chang}$-algebras. The answer is that the variety of BL$_\text{Chang}$-algebras is strictly contained in $P_0$. In fact, an easy check shows that a BL-chain is perfect if and only if the first component of its ordinal sum is a perfect MV-chain. Hence we have:
    \begin{itemize}
    \item Every BL$_\text{Chang}$-chain is a perfect BL-chain.
    \item There are perfect BL-chains that are not BL$_\text{Chang}$-chains: an example is given by $C\oplus [0,1]_\text{\L}$.
    \end{itemize}
    Now, since the variety of BL$_\text{Chang}$-algebras is generated by its chains (like any variety of BL-algebras, see \cite{haj}), then we get the result.

Finally note that \ref{eq:p0} is equivalent to $2(x^2)=(2x)^2$: hence, differently to what happens over {\L} (see \Cref{rem:2} ), the equations $2(x^2)=(2x)^2$ and $\overline{2}(x^2)=(\overline{2}x)^2$ are not equivalent, over BL. 
\end{itemize}
\end{remark}
\subsection{Subdirectly irreducible and simple algebras}
We begin with a general result about Wajsberg hoops.
\begin{theorem}[{\cite[Corollary 3.11]{f}}]
Every subdirectly irreducible Wajsberg hoop is totally ordered.
\end{theorem}
As a consequence, we have:
\begin{corollary}\label{cor:pssubir}
Every subdirectly irreducible pseudo-perfect Wajsberg hoop is totally ordered.
\end{corollary}
We now move to simple algebras.

It is shown in \cite[Theorem 1]{tur} that the simple BL-algebras coincide with the simple MV-algebras, that is, with the subalgebras of
$[0,1]_\text{\L}$ (see \cite[Theorem 3.5.1]{mun}). Therefore we have:
\begin{theorem}
The only simple BL$_\text{Chang}$-algebra is the two elements boolean algebra $\mathbf{2}$.
\end{theorem}
An easy consequence of this fact is that the only simple {\L}$_\text{Chang}$-algebra is $\mathbf{2}$.
\subsection{Completeness}\label{sec:propcompl}
We begin with a result about pseudo-perfect Wajsberg hoops.
\begin{theorem}
The class $pMV$ of $0$-free reducts of perfect MV-chains generates $ps\mathbb{WH}$.
\end{theorem}
\begin{proof}
From \Cref{teo:perfrot,teo:chainpswh} it is easy to check that the variety generated by $pMV$ contains all the totally ordered pseudo-perfect Wajsberg hoops.

From these facts and \Cref{cor:pssubir}, we have that $pMV$ must be generic for $ps\mathbb{WH}$.
\end{proof}
\begin{theorem}[\cite{dist}]\label{teo:dist}
Let L be an axiomatic extension of BL, then L enjoys the finite strong completeness w.r.t a class $K$ of L-algebras iff every countable L-chain is partially embeddable into $K$.
\end{theorem}
As shown in \cite{haj} product logic enjoys the finite strong completeness w.r.t $[0,1]_\Pi$ and hence every countable product chain is partially embeddable into $[0,1]_\Pi\simeq\mathbf{2}\oplus (0,1]_C$, with $(0,1]_C$ being the standard cancellative hoop (i.e. the $0$-free reduct of $[0,1]_\Pi\setminus\{0\}$). Since every totally ordered product chain is of the form $\mathbf{2}\oplus\mathcal{A}$, where $\mathcal{A}$ is a cancellative hoop (see \cite{eghm}), it follows that:
\begin{proposition}\label{prop:cancemb}
Every countable totally ordered cancellative hoop partially embeds into $(0,1]_C$.
\end{proposition}
\begin{theorem}\label{teo:perfho}
Every countable perfect MV-chain partially embeds into $\mathcal{V}=(0,1]^*_C$ (i.e. the disconnected rotation of $(0,1]_C$).
\end{theorem}
\begin{proof}
Immediate from \Cref{prop:cancemb} and \Cref{teo:perfrot}.
\end{proof}
\begin{corollary}\label{cor:comp}
The logic {\L}$_\text{Chang}$ is finitely strongly complete w.r.t. $\mathcal{V}$.
\end{corollary}
\begin{theorem}\label{teo:blccomp}
BL$_\text{Chang}$ enjoys the finite strong completeness w.r.t. $\omega\mathcal{V}$. As a consequence, the variety of BL$_\text{Chang}$-algebras is generated by the class of all ordinal sums of perfect MV-chains and hence is the smallest variety to contain this class of algebras.
\end{theorem}
\begin{proof}
Thanks to \Cref{teo:dist} it is enough to show that every countable BL$_\text{Chang}$-chain partially embeds into $\omega\mathcal{V}$ (i.e. the ordinal sum of ``$\omega$ copies'' of $\mathcal{V}$). This fact, however, follows immediately from \Cref{prop:cancemb} and \Cref{teo:chainstruct,teo:perfho}.
\end{proof}
But we cannot obtain a stronger result: in fact
\begin{theorem}
BL$_\text{Chang}$ is not strongly complete w.r.t. $\omega\mathcal{V}$.
\end{theorem}
\begin{proof}
Suppose not: from the results of \cite[Theorem 3.5]{dist} this is equivalent to claim that every countable BL$_\text{Chang}$-chain embeds into $\omega\mathcal{V}$. But, this would imply that every countable totally ordered cancellative hoop embeds into $(0,1]_C$: this means that every countable product-chain embeds into $[0,1]_\Pi$, that is product logic is strongly complete w.r.t $[0,1]_\Pi$. As it is well known (see \cite[Corollary 4.1.18]{haj}), this is false.
\end{proof}
With an analogous proof we obtain
\begin{theorem}
{\L}$_\text{Chang}$ is not strongly complete w.r.t. $\mathcal{V}$
\end{theorem}
However, thanks to \cite[Theorem 3]{monchain} we can claim
\begin{theorem}\label{teo:strcompl}
There exist a {\L}$_\text{Chang}$-chain $\mathcal{A}$ and a BL$_\text{Chang}$-chain $\mathcal{B}$ such that {\L}$_\text{Chang}$ is strongly complete w.r.t. $\mathcal{A}$ and BL$_\text{Chang}$ is strongly complete w.r.t. $\mathcal{B}$.
\end{theorem}
\begin{problem}
Which can be some concrete examples of such $\mathcal{A}$ and $\mathcal{B}$ ?
\end{problem}

\section{First-order logics}\label{sec:first}
We assume that the reader is acquainted with the formalization of first-order logics, as developed in \cite{haj, ch}.

Briefly, we work with (first-order) languages without equality, containing only predicate and constant symbols: as quantifiers we have $\forall$ and $\exists$. The notions of terms and formulas are defined inductively like in classical case.

As regards to semantics, given an axiomatic extension L of BL we restrict to L-chains: the first-order version of L is called L$\forall$ (see \cite{haj, ch} for an axiomatization). A first-order $\mathcal{A}$-interpretation ($\mathcal{A}$ being an L-chain) is a structure $\mathbf{M}=\lag M, \{r_P\}_{p\in \mathbf{P}}, \{m_c\}_{c\in \mathbf{C}}\rog$, where $M$ is a non-empty set, every $r_P$ is a fuzzy $ariety(P)$-ary relation, over $M$, in which we interpretate the predicate $P$, and every $m_c$ is an element of $M$, in which we map the constant $c$.

Given a map $v:\, VAR\to M$, the interpretation of $\lVert\varphi\rVert_{\mathbf{M}, v}^\mathcal{A}$ in this semantics is defined in a Tarskian way: in particular the universally quantified formulas are defined as the infimum (over $\mathcal{A}$) of truth values, whereas those existentially quantified are evaluated as the supremum. Note that these $\inf$ and $\sup$ could not exist in $\mathcal{A}$: an $\mathcal{A}$-model $\mathbf{M}$ is called \emph{safe} if $\lVert\varphi\rVert^\mathcal{A}_{\mathbf{M}, v}$ is defined for every $\varphi$ and $v$.

A model is called \emph{witnessed} if the universally (existentially) quantified formulas are evaluated by taking the minimum (maximum) of truth values in place of the infimum (supremum): see \cite{witn, ch1, ch} for details.

The notions of soundness and completeness are defined by restricting to safe models (even if in some cases it is possible to enlarge the class of models: see \cite{bm}): see \cite{haj, ch, ch1} for details.
\paragraph*{}
We begin with a positive result about \L$_\text{Chang}\forall$.
\begin{definition}
Let L be an axiomatic extension of BL. With L$\forall^w$ we define the extension of L$\forall$ with the following axioms
\begin{align}
\tag*{(C$\forall$)}&(\exists y)(\varphi(y)\to (\forall x)\varphi(x))\label{cforall}\\
\tag*{(C$\exists$)}&(\exists y)((\exists x)\varphi(x)\to\varphi(y))\label{cexist}.
\end{align}
\end{definition}

\begin{theorem}[{\cite[Proposition 6]{ch1}}]
\L$\forall$ coincides with \L$\forall^w$, that is \\\L$\forall\vdash$\emph{\ref{cforall},\ref{cexist}}.
\end{theorem}
An immediate consequence is:
\begin{corollary}\label{cor:witnluk}
Let L be an axiomatic extension of \L. Then L$\forall$ coincides with L$\forall^w$.
\end{corollary}
\begin{theorem}[{\cite[Theorem 8]{ch1}}]\label{teo:witncompl}
Let L be an axiomatic extension of BL. Then L$\forall^w$ enjoys the strong witnessed completeness with respect to the class $K$ of L-chains, i.e.
\begin{equation*}
T\vdash_{L\forall^w}\varphi\quad\text{iff}\quad \lVert\varphi\rVert_\mathbf{M}^\mathcal{A}=1,
\end{equation*}
for every theory $T$, formula $\varphi$, algebra $\mathcal{A}\in K$ and witnessed $\mathcal{A}$-model $\mathbf{M}$ such that $\lVert\psi\rVert_\mathbf{M}^\mathcal{A}=1$ for every $\psi\in T$.
\end{theorem}
\begin{lemma}[{\cite[Lemma 1]{monchain}}]\label{lem:witn}
Let L be an axiomatic extension of BL, let $\mathcal{A}$ be an L-chain, let $\mathcal{B}$ be
an L-chain such that $A\subseteq B$ and let $\mathbf{M}$ be a witnessed $\mathcal{A}$-structure. Then for every
formula $\varphi$ and evaluation $v$, we have
$\lVert\varphi\rVert^\mathcal{A}_{\mathbf{M},v}=\lVert\varphi\rVert^\mathcal{B}_{\mathbf{M},v}$.
\end{lemma}
\begin{theorem}
There is a \L$_\text{Chang}$-chain such that \L$_\text{Chang}\forall$ is strongly complete w.r.t. it. More in general, every \L$_\text{Chang}$-chain that is strongly complete w.r.t \L$_\text{Chang}$ is also strongly complete w.r.t. \L$_\text{Chang}\forall$.
\end{theorem}
\begin{proof}
An adaptation of the proof for the analogous result, given in \cite[Theorem 16]{monchain}, for \L$\forall$.

From \Cref{teo:strcompl} we know that there is a \L$_\text{Chang}$-chain $\mathcal{A}$ strongly complete w.r.t. \L$_\text{Chang}$: from \cite[Theorem 3.5]{dist} this is equivalent to claim that every countable \L$_\text{Chang}$-chain embeds into $\mathcal{A}$.
We show that $\mathcal{A}$ is also strongly complete w.r.t. \L$_\text{Chang}\forall$.

Suppose that $T\not\vdash_{\text{\L}_\text{Chang}\forall}\varphi$. Thanks to \Cref{cor:witnluk} and \Cref{teo:witncompl} there is a countable \L$_\text{Chang}$-chain $\mathcal{C}$ and a witnessed $\mathcal{C}$-model $\mathbf{M}$ such that $\lVert\psi\rVert_\mathbf{M}^\mathcal{C}=1$, for every $\psi\in T$, but $\lVert\varphi\rVert_\mathbf{M}^\mathcal{C}<1$. Finally, from \Cref{lem:witn} we have that $\lVert\psi\rVert_\mathbf{M}^\mathcal{A}=1$, for every $\psi\in T$ and $\lVert\varphi\rVert_\mathbf{M}^\mathcal{A}=\lVert\varphi\rVert_\mathbf{M}^\mathcal{C}<1$: this completes the proof.
\end{proof}
For BL$_\text{Chang}\forall$, however, the situation is not so good.
\begin{theorem}
BL$_\text{Chang}\forall$ cannot enjoy the completeness w.r.t. a single BL$_\text{Chang}$-chain.
\end{theorem}
\begin{proof}
The proof is an adaptation of the analogous result given in \cite[Theorem 17]{monchain} for BL$\forall$.

Let $\mathcal{A}$ be a BL$_\text{Chang}$-chain: call $\mathcal{A}_0$ its first component. We have three cases
\begin{itemize}
\item $\mathcal{A}_0$ is finite: from \Cref{teo:chainstruct} we have that $\mathcal{A}_0=\mathbf{2}$ and hence $\mathcal{A}\models (\neg\neg x)\rightarrow (\neg\neg x)^2$. However $\mathcal{V}\not\models (\neg\neg x)\rightarrow (\neg\neg x)^2$, where $\mathcal{V}$ is the chain introduced in \Cref{sec:propcompl}, and hence $\mathcal{A}$ cannot be complete w.r.t. BL$_\text{Chang}\forall$.
\item $\mathcal{A}_0$ is infinite and dense. As shown in \cite[Theorem 17]{monchain} the formula \\$(\forall x)\neg\neg P(x)\to \neg\neg (\forall x) P(x)$ is a tautology in every BL-chain whose first component is infinite and densely ordered: hence we have that $\mathcal{A}\models (\forall x)\neg\neg P(x)\to \neg\neg (\forall x) P(x)$. However it is easy to check that this formula fails in $[0,1]_G$: take a $[0,1]_G$-model $\mathbf{M}$ with $M=(0,1]$ and such that $r_P(m)=m$. Hence, from \Cref{cor:blchacont}, it follows that BL$_\text{Chang}\forall\not\vdash (\forall x)\neg\neg P(x)\to \neg\neg (\forall x) P(x)$.
\item $\mathcal{A}_0$ is infinite and not dense. As shown in \cite[Theorem 17]{monchain} the formula $(\forall x)\neg\neg P(x)\to \neg\neg (\forall x) P(x)\vee \neg (\forall x) P(x)\to ((\forall x) P(x))^2$ is a tautology in every BL-chain whose first component is infinite and not densely ordered: hence we have that $\mathcal{A}\models (\forall x)\neg\neg P(x)\to \neg\neg (\forall x) P(x)\vee \neg (\forall x) P(x)\to ((\forall x) P(x))^2$. Also in this case, however, this formula fails in $[0,1]_G$, using the same model $\mathbf{M}$ of the previous case.
\end{itemize}
\end{proof}
\bibliography{blchang}
\bibliographystyle{amsalpha}
\end{document}